\documentclass[12pt]{amsart}

\usepackage{amsmath}
\usepackage{amsfonts}
\usepackage{amsthm}

\usepackage[cmtip,arrow]{xy}
\usepackage{pb-diagram,pb-xy}

\newtheorem{theorem}{Theorem}
\newtheorem{corollary}{Corollary}

\newtheorem{proposition}{Proposition}
\newtheorem{definition}{Definition}

\newtheorem{remark}{Remark}

\begin{document}
\title{Gorenstein projective precovers and finitely presented modules}

\author{Sergio Estrada}
\address{S.E. \ Departamento de Matem\'aticas\\ Universidad de Murcia\\ Murcia 30100, Spain}
\email[Sergio Estrada]{sestrada@um.es}
\author{Alina Iacob}
\address{A.I. \ Department of Mathematical Sciences \\
         Georgia Southern University \\
         Statesboro (GA) 30460-8093 \\ U.S.A.}
\email[Alina Iacob]{aiacob@GeorgiaSouthern.edu}

\thanks{The first author was partly supported by grant
  PID2020-113206GB-I00 funded by MCIN/AEI/10.13039/ 501100011033 and by grant 22004/PI/22 funded by Fundaci\'on S\'eneca.}
\maketitle

\begin{abstract}
The existence of the Gorenstein projective precovers over arbitrary rings is an open question. It is known that if  the ring has
finite Gorenstein global dimension, then every module has a Gorenstein projective precover. We prove here a "reduction" property - we show that, over any ring, it suffices
to consider finitely presented modules: if there exists a nonnegative integer $n$ such that every finitely presented module
has Gorenstein projective dimension $\le n$, then the class of Gorenstein
projective modules is special precovering.
\end{abstract}

\section{introduction}
The Gorenstein projective modules were introduced by Enochs and Jenda in 1995 (\cite{enochs:95:gorenstein}). They are the cycles of the exact complexes of projective modules
that remain exact when applying a functor $Hom(-, P)$, with $P$ any projective module. Such a complex is called a totally acyclic complex. We use $K_{tac}(Proj)$ to denote the homotopy category of totally acyclic complexes of projective modules. This is a full subcategory of that of complexes of projective modules, $K(Proj)$.\\
We will use $\mathcal{GP}$ to
denote the class of Gorenstein projective modules. Together with the Gorenstein injective and with the Gorenstein flat modules, they are the fundamental objects of Gorenstein
Homological Algebra. The Gorenstein methods have proved to be very useful, but they can only be used as long as the Gorenstein resolutions exist. The existence of the
Gorenstein resolutions over Gorenstein rings was proved by Enochs and Jenda (\cite{enochs:00:relative}). Then J{\o}rgensen proved (2007) that when $R$ is a commutative
noetherian
ring with a dualizing complex, the inclusion functor $e: K_{tac} (Proj) \rightarrow K(Proj)$
has a right adjoint. Using this adjoint he proved the existence of the Gorenstein projective resolutions over such rings. More recently, Murfet and
Salarian proved their existence over commutative noetherian rings of finite Krull dimension (in 2011). In \cite{iacob:14:gorproj}, we extended their result: we proved that
the class of Gorenstein projective modules is special precovering over any right coherent and left $n$-perfect ring. Consequently, over such rings, every module has a
Gorenstein projective resolution\\
But the existence of the Gorenstein projective resolutions over arbitrary rings is still an open question. It is known that if the ring has finite Gorenstein global dimension
(i.e. if there is a nonnegative integer $n$ such that any $R$-module has Gorenstein projective dimension $\le n$), then every module has a Gorenstein projective resolution.
We prove here a "reduction" property. We show (Theorem \ref{thm:Gpprecovering}) that, over any ring, it suffices to consider finitely presented modules: if there exists a nonnegative integer n
such that every
finitely presented module has Gorenstein projective dimension $\le n$, then the class of Gorenstein projective modules is special precovering. Previously, the result was only
known over two sided noetherian rings (\cite{enochs:00:relative}, Proposition 12.3.1).\\
As an application, we show (Corollary \ref{finitistic}) that under certain conditions the validity of the Second
Finitistic Dimension Conjecture implies that the class of Gorenstein
projective modules is special precovering.\\
We also consider the class of Ding projective modules. These modules were introduced in \cite{ding and mao 08}, as generalizations of the Gorenstein projective modules. They
are the cycles of the exact complexes of projective modules that remain exact when applying a functor $Hom(-, F)$, with $F$ any flat module. It is known (see for example
\cite{iacob:2020}) that the class of Ding projectives is special precovering over any coherent ring. But, just as in the case of the Gorenstein projectives, the existence of
the Ding projective precovers over arbitrary rings is still an open question. We show that this problem can also be reduced to finitely presented modules: if there exists a
nonnegative integer $n$ such that every
finitely presented module has Ding projective dimension $\le n$, then the class of Ding projective modules is special precovering.

\section{preliminaries}
Throughout this paper $R$ will be an associative ring with identity and all modules are left $R$-modules.

\begin{definition}
A module $M$ is strongly FP-injective if, for any finitely presented module $T$, $Ext^i(T,M)=0$ for all $i \ge 1$.
\end{definition}

 We use $\mathcal{SFPI}$ to denote the class of strongly FP-injective modules, and $\mathcal{C}$ to denote its left orthogonal class, $\mathcal{C} = ^\bot \mathcal{SFPI}$.
 The modules in $\mathcal{C}$ are named \emph{weakly FP-projective modules} in \cite{BHP:2022}.

We note that if the ring $R$ is left coherent, then the class of strongly FP-injective modules is simply that of FP-injective modules, $\mathcal{FPI}$ (see, for instance
\cite{LGO:2017}, Theorem 4.2).\\

The following results are from Emmanouil and Kaperonis, \cite{emmanouil:kaperonis}, and Li, Guan and Ouyang, \cite{LGO:2017}:\\

\begin{proposition} (\cite{LGO:2017}, Theorem 3.4) Over any ring $R$, $(\mathcal{C}, \mathcal{SFPI})$ is a complete and hereditary cotorsion pair in $R-Mod$.
\end{proposition}
\begin{proof}
The cotorsion pair $(\mathcal{C}, \mathcal{SFPI})$ is complete and hereditary since it is generated by a representative set of the finitely presented modules and their syzygies (this follows from \cite{ET}, Theorem 10).
\end{proof}

\begin{remark}
If $R$ is a left coherent ring, then $\mathcal{C}$ is the class of FP-projective modules.
\end{remark}

This follows from the fact that $(FP-Proj, \mathcal{FPI})$ is a cotorsion pair for any ring, and from the fact that $\mathcal{SFPI} = \mathcal{FPI}$ over a left coherent ring.
Thus $FP-Proj = ^\bot \mathcal{FPI} = ^\bot \mathcal{SFPI} = \mathcal{C}$ when $R$ is left coherent.

\begin{proposition} (\cite{emmanouil:kaperonis}, Corollary 4.9)
Every acyclic complex of projective modules is in $\widetilde{\mathcal{C}}$ (where $\widetilde{\mathcal{C}}$ is the class of acyclic complexes with all cycles from the class
$\mathcal{C}$).
\end{proposition}

Our results focus on Gorenstein projective modules, and on Ding projective modules, so we recall the definitions:\\
\begin{definition}
A module $G$ is Gorenstein projective if there exists an exact complex of projective modules $P = \ldots \rightarrow P_1 \rightarrow P_0 \rightarrow P_{-1} \rightarrow \ldots
$ such that $Hom(P, T)$ is exact for any projective module $T$, and such that $G = Z_0(P)$.
\end{definition}

 We recall that a Gorenstein projective precover of a module $M$ is a homomorphism
$g: G \rightarrow M$ with $G$ a Gorenstein projective module, and
with the property that any homomorphism $h: G' \rightarrow M$, from a Gorenstein projective
module $G'$ to $M$ factors through $g$ ($h =
gu$ for some $u \in Hom(G', G)$).

\[
\begin{diagram}
\node{}\node{G'}\arrow{sw,t,..}{u}\arrow{s,r}{h}\\
\node{G}\arrow{e,t}{g}\node{M}
\end{diagram}
\]

Such a Gorenstein projective precover $g: G \rightarrow M$ is said to be \emph{special} if $Ext^1(G', Ker(g)) =0$ for any Gorenstein projective module $G'$.\\

The Ding projective modules were introduced in \cite{ding and mao 08}, where they were called strongly Gorenstein flat modules. Later they were renamed Ding projective
modules (in \cite{gillespie-ding-modules}). The class of Ding projective modules is known to be precovering over coherent rings. But, just as in the case of the Gorenstein
projectives, the existence of the Ding projective precovers over arbitrary rings is an open question.\\
\begin{definition}
A module $M$ is Ding projective if there exists an exact complex of projective modules $P = \ldots \rightarrow P_1 \rightarrow P_0 \rightarrow P_{-1} \rightarrow \ldots $
such that $Hom(P,F)$ is exact for any flat module $F$, and such that $M = Z_0(P)$.
\end{definition}

We will use $\mathcal{GP}$ to denote the class of Gorenstein projective modules, and $\mathcal{DP}$ for the class of Ding projective modules. It is immediate from the
definitions that $\mathcal{DP} \subseteq \mathcal{GP}$.\\

 Since the Ding projective modules, as well as the Gorenstein projectives are cycles of exact complexes of projective modules, and, by Proposition 2, any such complex is in
 $\widetilde{\mathcal{C}}$, we obtain:\\
 \begin{remark}
 $\mathcal{DP} \subseteq \mathcal{GP} \subseteq \mathcal{C}$, where $\mathcal{DP}$ is the class of Ding projective modules, and $\mathcal{GP}$ is that of Gorenstein
 projective modules.
 \end{remark}




The Ding projective (special) precovers are defined in a similar manner with the Gorenstein projective ones, by replacing the class of Gorenstein projectives with that of Ding projective modules in the definition.

The existence of Gorenstein projective precovers (Ding projective precovers respectively) allows defining Gorenstein projective resolutions (Ding projective resolutions
respectively). A Gorenstein projective resolution of a module $M$ is a complex $$ \ldots \rightarrow G_1 \rightarrow G_0
\rightarrow M \rightarrow 0$$ such that
$G_0 \rightarrow M$ and each $G_i \rightarrow Ker (G_{i-1}
\rightarrow G_{i-2})$ for $i \ge 1$ are Gorenstein projective
precovers. The Ding projective resolutions are defined in a similar manner. \\Thus the fact that every module $M$ over a ring $R$ has a
Gorenstein projective (Ding projective) resolution is equivalent to the class of
Gorenstein projective modules (Ding projective modules) being a precovering class over $R$.\\

\section{results}

We start by showing that both the problem of the existence of special Gorenstein projective precovers and that of the existence of special Ding projective precovers can be
reduced to their existence for modules in the class $\mathcal{C}$. Both results follow from the following:\\

\begin{theorem}
Let $(\mathcal{C}, \mathcal{E})$ be a complete hereditary cotorsion pair in $R-Mod$, and let $\mathcal{D}$ be a subclass of $\mathcal{C}$. Then $\mathcal{D}$ is special precovering if and only if every module in $\mathcal{C}$ has a special $\mathcal{D}$-precover.
\end{theorem}

 \begin{proof}
 One implication is immediate: if $\mathcal{D}$ is special precovering then every module in $\mathcal{C}$ has a special $\mathcal{D}$-precover.\\
 For the converse: let $M$ be any $R$-module. Since $(\mathcal{C}, \mathcal{E})$ is a complete cotorsion pair, there is a short exact sequence
 $0 \rightarrow S \rightarrow C \rightarrow M \rightarrow 0$ with $C \in \mathcal{C}$ and with $S \in \mathcal{E}$.\\
 By hypothesis, there is an exact sequence $0 \rightarrow X \rightarrow D \rightarrow C  \rightarrow 0$ with $D \in \mathcal{D}$ and with $X \in \mathcal{D}^\bot$. \\

 Form the pull back diagram:\\

\[
\begin{diagram}
\node{}\node{0}\arrow{s}\node{0}\arrow{s}\\
\node{}\node{X}\arrow{s}\arrow{e,=}\node{X}\arrow{s}\\
\node{0}\arrow{e}\node{V}\arrow{s}\arrow{e}\node{D}\arrow{s}\arrow{e}\node{M}\arrow{s,=}\arrow{e}\node{0}\\
\node{0}\arrow{e}\node{S}\arrow{e}\node{C}\arrow{e}\node{M}\arrow{e}\node{0}
\end{diagram}
\]

 The exact sequence $0 \rightarrow X \rightarrow V \rightarrow S \rightarrow 0$ with $X \in \mathcal{D}^\bot$, and with $S \in \mathcal{C}^\bot \subseteq \mathcal{D}^\bot$
 (because $\mathcal{D} \subseteq \mathcal{C}$) gives that $V \in \mathcal{D}^\bot$.\\

 The exact sequence $0 \rightarrow V \rightarrow D \rightarrow M \rightarrow 0$ with $D \in \mathcal{D}$ and with $V \in \mathcal{D}^\bot$, shows that $D
 \rightarrow M$ is a special $\mathcal{D}$-precover.
\end{proof}

 Since $(\mathcal{C}, \mathcal{SFPI})$ is a complete hereditary cotorsion pair, and $\mathcal{DP} \subseteq \mathcal{C}$, we obtain:\\

  \begin{corollary}
 The class of Ding projective modules is special precovering if and only if every module in $\mathcal{C}$ has a special Ding projective precover.
 \end{corollary}

 Another application of Theorem 1 (for $\mathcal{GP} \subseteq \mathcal{C}$) gives:\\

 \begin{corollary}
 The class of Gorenstein projective modules is special precovering if and only if every module in $\mathcal{C}$ has a special Gorenstein projective precover.
 \end{corollary}

 It is known (\cite{Hol}, Theorem 2.10) that every module of finite Gorenstein projective dimension has a special Gorenstein projective precover. So we obtain:\\
 \begin{theorem}\label{GPspecial precovering:reduction}
 If every module from the class $\mathcal{C}$ has finite Gorenstein projective dimension, then the class of Gorenstein projective modules is special precovering in $R-Mod$.
 \end{theorem}

 We prove that a sufficient condition for $\mathcal{GP}$ being a special precovering class is
 having an upper bound $n$ for $Gpd$ $C$ for every finitely presented module $C$.\\
  We will use the following results:\\

 \begin{proposition}\label{prop:Gpdcloseddirectsummand}
 Let $R$ be any ring. The class of modules of finite Gorenstein projective dimension is closed under direct summands.
 \end{proposition}

\begin{proof}
Let $G$ be a module with $Gpd$ $G = n \ge 0$ and $G = M \oplus M'$. Let $$0\to K_n\to P_{n-1}\to \ldots \to P_0\to M\to 0 $$ and $$0\to K'_n\to P'_{n-1}\to \ldots \to P'_0\to M'\to 0 $$ be exact sequences, where $P_0,\ldots, P_{n-1}$ and $P_0',\ldots, P_{n-1}'$ are projective modules. Since $Gpd$ $G = n$, we get that the module $K_n\oplus K'_n$ is Gorenstein projective, but then by Holm \cite{Hol}, Theorem 2.5, the modules $K_n$ and $K_n'$ are Gorenstein projective. Therefore we get that  $Gpd$ $M \leq n$ and  $Gpd$ $M' \leq n$.
\end{proof}

\begin{theorem}\label{thm:Gpprecovering}
  If there exists a nonnegative integer $n$ such that every finitely presented module
 has Gorenstein projective dimension $\le n$, then the class of Gorenstein projective modules is special precovering.
\end{theorem}
\begin{proof} First we recall that $M\in\mathcal{SFPI}\iff Ext^i(T,M)=0$ for each finitely presented module $T$ and every $i\geq 1$. Now, $0=Ext^i(T,M)\simeq
Ext^1(\Omega^{i-1}(T),M)$, where $\Omega^{i-1}(T)$ is the $(i-1)$th syzygy module of $T$ (where $\Omega^{0}(T)=T$). Therefore, the cotorsion pair
$(\mathcal{C}, \mathcal{SFPI})$ is generated by the set $\mathcal S$ of representatives of modules $\Omega^j(T)$, with $T$  finitely presented and $j\geq 0$. Now, by
hypothesis, $Gpd\ T \le n$ for each finitely presented module $T$. This immediately yields, in particular, (see for example \cite{Hol}, Proposition 2.18) that $Gpd$
$\Omega^j(T)\le n$ for each $j\geq 0$, i.e. $Gpd$ $S\le n$, for each $S\in \mathcal S$. Now, by \cite{ET}, we have that $\mathcal{C}=\textrm{add}(\textrm{Filt}(\mathcal S))$,
i.e. every module in $\mathcal C$ is a direct summand of an $\mathcal S$-filtered module. Now, by \cite{EIJ}, Theorem 3.4, every $\mathcal S$-filtered module has Gorenstein
projective dimension $\leq n$. Finally, by Proposition \ref{prop:Gpdcloseddirectsummand}, we get that every module in $\mathcal C$ has Gorenstein projective dimension $\leq n$. Therefore, by Theorem \ref{GPspecial precovering:reduction},
the class $\mathcal{GP}$ is special precovering.
\end{proof}

\begin{remark}
It was already proved by Enochs and Jenda (\cite{enochs:00:relative}, Theorem 12.3.1), that if $R$ is noetherian on both sides and each finitely generated module (left or
right) has Gorenstein projective dimension $\le n$ then the class of Gorenstein projective modules is special precovering (because $R$ is $n$-Gorenstein in this case). From
this point of view, the previous statement might be seen as a one-sided analogue of the previous statement for arbitrary rings.
\end{remark}
A similar argument gives:\\
\begin{theorem}
 If there exists a nonnegative integer $n$ such that every finitely presented module
 has Ding projective dimension $\le n$, then the class of Ding projective modules is special precovering.
\end{theorem}

Since the two pairs $(\mathcal{GP},\mathcal{GP}^\perp)$ and $(\mathcal{DP},\mathcal{DP}^\perp)$ are always hereditary cotorsion pairs (Cort\'es-Izurdiaga and Saroch,
\cite{cortes.saroch}, Corollary 3.4) the previous two results immediately yield the following corollary.
\begin{corollary}
The following statements hold:
\begin{enumerate}
\item If there exists a nonnegative integer $n$ such that every finitely presented module
 has Gorenstein projective dimension $\le n$, then the cotorsion pair $(\mathcal{GP},\mathcal{GP}^\perp)$ is complete hereditary.
\item If there exists a nonnegative integer $n$ such that every finitely presented module
 has Ding projective dimension $\le n$, then the cotorsion pair $(\mathcal{DP},\mathcal{DP}^\perp)$ is complete hereditary.
 \end{enumerate}
\end{corollary}
\begin{remark}
It is known that $(\mathcal{GP},\mathcal{GP}^\perp)$ is a complete hereditary cotorsion pair for rings $R$ such that $Ggldim(R)<\infty$ (see, for example, \cite{EEJR},
Theorem 2.26, and \cite{BM}, Theorem 1.1), where
$Ggldim(R) = sup \{Gpd M$, $M$ is a module $\}$. The previous corollary establishes the same statement for rings with
$sup \{Gpd M$, $M$ is finitely presented $\}<\infty$.
\end{remark}

We recall that the small finitistic dimension of a ring $R$ is defined to be $fpd(R) = sup \{p.d. M$, $M$ is finitely generated, with $p.d. M < \infty \}$. \\
We also recall that the Second Finitistic Dimension Conjecture states that $ fpd(R) < \infty$.\\
It is known that, if the ring $R$ is left noetherian, then the small finitistic dimension can be computed by replacing the class of finitely generated modules of finite
projective dimension with its Gorenstein counterpart - the class of finitely generated modules of finite Gorenstein projective dimension (see for example \cite{wang.li.hu},
Lemma 4.2, or \cite{moradifar.saroch}, page 4).

\begin{corollary}\label{finitistic}
Over a left noetherian ring $R$ such that every finitely generated module has finite Gorenstein projective dimension, the validity of the Second Finitistic Dimension Conjecture implies that the class of Gorenstein projective modules is special precovering.
\end{corollary}

\begin{proof}
The validity of the Second Finitistic Dimension Conjecture means that there is a nonnegative integer $n$ such that every finitely generated module $M$ has Gorenstein
projective dimension $\le n$. By Theorem \ref{thm:Gpprecovering}, the class of Gorenstein projective modules is special precovering.
\end{proof}
\begin{remark}
By \cite{EIJ}, Proposition 3.5, and \cite{cortes}, Corollary 3.6, if $R$ is left noetherian and every finitely generated module has finite Gorenstein projective dimension $\leq n$, then $Ggldim(R)<\infty$, and so the class of Gorenstein projective modules is special precovering. Theorem \ref{thm:Gpprecovering} provides with more elementary proof of this fact. In addition, there are rings of infinite Gorenstein global dimension such that every finitely presented module has finite Gorenstein projective dimension $\leq n$; every von Neumann regular ring of infinite global dimension exemplifies this (see \cite{CET}, Remark 4.7, for a concrete example).
\end{remark}

\begin{center}{\bf Acknowledgements}
\end{center}
{\par \noindent
The authors wish to thank Manuel Cort\'es-Izurdiaga and Ioannis Emmanouil for useful and pertinent comments to an earlier version of this manuscript.}

\end{document}